\documentclass[english]{article}

\usepackage{lmodern}
\usepackage[a4paper]{geometry}

\usepackage[utf8]{inputenc}
\usepackage[T1]{fontenc}

\usepackage[english]{babel}
\usepackage{graphicx}
\DeclareGraphicsExtensions{.pdf,.png,.jpg}
\usepackage{amsmath}
\usepackage{amssymb}
\usepackage{amsthm}
\usepackage[dvipsnames]{xcolor}

\usepackage{wrapfig}

\usepackage{tikz,pgf}
\usepackage{tkz-euclide}
\usepackage{tkz-graph}
\usetkzobj{all} 
\usepackage{hyperref}

\usepackage{authblk}
\newcommand{\RR}{\mathbb{R}}

\tikzstyle{vertex}=[circle, draw, fill=black, inner sep=0pt, minimum size=4pt]
\tikzstyle{smallvertex}=[circle, draw, fill=black, inner sep=0pt, minimum size=2pt]
\tikzstyle{edge}=[line width=1.5pt,black!50!white]
\tikzstyle{dedge}=[line width=1.5pt,black!50!white, densely dashed]
\tikzstyle{bdedge}=[line width=1.5pt,NavyBlue, densely dashed]
\tikzstyle{rdedge}=[line width=1.5pt,Red, densely dashed]
\tikzstyle{redge}=[line width=1.5pt,Red]
\tikzstyle{bedge}=[line width=1.5pt,NavyBlue]
\tikzstyle{edgeq}=[edge,gray!60,densely dashed]
\tikzstyle{lnode}=[circle,white,draw=black!80!white,fill=black!80!white,inner sep=0.5pt, font=\scriptsize]

\def\Red{\textcolor{red}}
\def\Blue{\textcolor{blue}}

\newtheorem{thm}{Theorem}
\newtheorem{defn}{Definition}
\newtheorem{lem}{Lemma}
\newtheorem{cor}{Corollary}

\newcommand{\maxSevenVert}{
			\coordinate(1) at (0, -1);
		      \coordinate (2) at (-1.9, 0);
		      \coordinate (3) at (-0.9, -0.3) ;
		      \coordinate (4) at (0.85, -0.3) ;
		      \coordinate (5) at (1.8,0.0) ;
		      \coordinate (6) at (-0.25, 0.25) ;
	   	      \coordinate (7) at (0,1) ;}

\newcommand{\ringEightVert}{
			\coordinate(1) at (0, -1);
		      \coordinate (2) at (-1.9, 0);
		      \coordinate (3) at (-0.9, -0.3) ;
		      \coordinate (4) at (0.8, -0.3) ;
		      \coordinate (5) at (1.9,0.0) ;
		      \coordinate (6) at (0.9, 0.3) ;
	   	      \coordinate (7) at (-0.8,0.3) ;
	   	      \coordinate (8) at (0,1) ;}

\newcommand{\maxEightVert}{
	\coordinate(1) at (0, -1);
	\coordinate (2) at (-1.9, 0);
	\coordinate (3) at (-0.9, -0.3) ;
	\coordinate (4) at (0.85, -0.3) ;
	\coordinate (5) at (1.8,0.0) ;
	\coordinate (6) at (-0.3, 0.3) ;
	\coordinate (7) at (-0.7,1) ;
	\coordinate (8) at (0.7,1) ;}

\newcommand{\HI}{
\begin{tikzpicture}[scale=0.8]
	\node[smallvertex] (1) at (0, 0) {};
	\node[smallvertex] (2) at (0.25, 0.4) {};
	\node[smallvertex] (3) at (0.7, 0.3) {};      
	\draw[gray!60,densely dashed] (0.3,0.2) circle (0.6cm);
	\draw[very thick,->] (0.3,-0.5) -- (0.3,-0.9);
	\begin{scope}[yshift=-1.8cm]
		\node[smallvertex] (1) at (0, 0) {};
		\node[smallvertex] (2) at (0.25, 0.4) {};
		\node[smallvertex] (3) at (0.7, 0.3) {}; 
		\node[smallvertex] (w) at (0.65, -0.55) {}; 
		\draw[gray!60,densely dashed] (0.3,0.2) circle (0.6cm);
		\draw[bedge] (1) to (w);
		\draw[bedge] (2) to (w);
		\draw[bedge] (3) to (w);
	\end{scope}
	\node (u) at (1.0,0) {};
	\node (u) at (-0.4,0) {};
\end{tikzpicture}
}

\newcommand{\HII}{
\begin{tikzpicture}[scale=0.8]
	\node[smallvertex] (1) at (0, 0) {};
	\node[smallvertex] (2) at (0.35, 0.5) {};
	\node[smallvertex] (3) at (0.7, 0.3) {};      
	\node[smallvertex] (4) at (0.3,0.0) {};  
	\draw[redge] (1) to (2);
	\draw[gray!60,densely dashed] (0.3,0.2) circle (0.6cm);
	\draw[very thick,->] (0.3,-0.5) -- (0.3,-0.9);
	\begin{scope}[yshift=-1.8cm]
		\node[smallvertex] (1) at (0, 0) {};
		\node[smallvertex] (2) at (0.35, 0.5) {};
		\node[smallvertex] (3) at (0.7, 0.3) {}; 
		\node[smallvertex] (4) at (0.3,0.0) {}; 
		\node[smallvertex] (w) at (0.65, -0.55) {}; 
		\draw[gray!60,densely dashed] (0.3,0.2) circle (0.6cm);
		\draw[bedge] (1) to (w);
		\draw[bedge] (2) to (w);
		\draw[bedge] (3) to (w);
		\draw[bedge] (4) to (w);
	\end{scope}
	\node (u) at (1.0,0) {};
	\node (u) at (-0.4,0) {};
\end{tikzpicture}
}

\newcommand{\HIIIx}{
\begin{tikzpicture}[scale=0.8]
	\node[smallvertex] (1) at (-0.05, 0.1) {};
	\node[smallvertex] (2) at (0.35, 0.5) {};
	\node[smallvertex] (3) at (0.7, 0.3) {}; 
	\node[smallvertex] (4) at (0.1,-0.2) {}; 
	\node[smallvertex] (5) at (0.0,0.45) {}; 
	\draw[redge] (1) to (3);
	\draw[redge] (2) to (4);
	\draw[gray!60,densely dashed] (0.3,0.2) circle (0.6cm);
	\draw[very thick,->] (0.3,-0.5) -- (0.3,-0.9);
	\begin{scope}[yshift=-1.8cm]
		\node[smallvertex] (1) at (-0.05, 0.1) {};
		\node[smallvertex] (2) at (0.35, 0.5) {};
		\node[smallvertex] (3) at (0.7, 0.3) {}; 
		\node[smallvertex] (4) at (0.1,-0.2) {}; 
		\node[smallvertex] (5) at (0.0,0.45) {}; 
		\node[smallvertex] (w) at (0.7, -0.5) {}; 
		\draw[gray!60,densely dashed] (0.3,0.2) circle (0.6cm);
		\draw[bedge] (1) to (w);
		\draw[bedge] (2) to (w);
		\draw[bedge] (3) to (w);
		\draw[bedge] (4) to (w);
		\draw[bedge] (5) to (w);
	\end{scope}
	\node (u) at (1.0,0) {};
	\node (u) at (-0.4,0) {};
\end{tikzpicture}
}

\newcommand{\HIIIv}{
\begin{tikzpicture}[scale=0.8]
	\node[smallvertex] (1) at (-0.05, 0.1) {};
	\node[smallvertex] (2) at (0.4, 0.55) {};
	\node[smallvertex] (3) at (0.7, 0.3) {}; 
	\node[smallvertex] (4) at (0.1,-0.2) {}; 
	\node[smallvertex] (5) at (0.0,0.45) {}; 
	\draw[redge] (1) to (2);
	\draw[redge] (2) to (4);
	\draw[gray!60,densely dashed] (0.3,0.2) circle (0.6cm);
	\draw[very thick,->] (0.3,-0.5) -- (0.65,-0.9);
	\begin{scope}[xshift=1.4cm]
		\node[smallvertex] (1) at (-0.05, 0.1) {};
		\node[smallvertex] (2) at (0.4, 0.55) {};
		\node[smallvertex] (3) at (0.7, 0.3) {}; 
		\node[smallvertex] (4) at (0.1,-0.2) {}; 
		\node[smallvertex] (5) at (0.0,0.45) {}; 
		\draw[redge] (1) to (5);
		\draw[redge] (5) to (3);
		\draw[gray!60,densely dashed] (0.3,0.2) circle (0.6cm);
		\draw[very thick,->] (0.3,-0.5) -- (-0.05,-0.9);
	\end{scope}
	\begin{scope}[yshift=-1.8cm, xshift=0.7cm]
		\node[smallvertex] (1) at (-0.05, 0.1) {};
		\node[smallvertex] (2) at (0.4, 0.55) {};
		\node[smallvertex] (3) at (0.7, 0.3) {}; 
		\node[smallvertex] (4) at (0.1,-0.2) {}; 
		\node[smallvertex] (5) at (0.0,0.45) {}; 
		\node[smallvertex] (w) at (0.7, -0.5) {}; 
		\draw[gray!60,densely dashed] (0.3,0.2) circle (0.6cm);
		\draw[bedge] (1) to (w);
		\draw[bedge] (2) to (w);
		\draw[bedge] (3) to (w);
		\draw[bedge] (4) to (w);
		\draw[bedge] (5) to (w);
	\end{scope}
\end{tikzpicture}
}

\newcommand{\cubeCoord}{\begin{scope}[xscale=1.2]
				\coordinate (1) at (-0.30,0.275883705367221);
				\coordinate (2) at (0.4,1.0);
				\coordinate (3) at (0.4,0.45);
				\coordinate (4) at (0.51261105436232,0.0) ;
				\coordinate (5) at (-0.3,0.784296289346965);
				\coordinate (6) at (1.2,0.784296289346965);
				\coordinate (7) at (1.2,0.275883705367221);
				\end{scope}}

\title{On the Maximal Number of Real Embeddings of Spatial Minimally Rigid  Graphs}

\author[1,2]{Evangelos Bartzos}
\author[1,2]{Ioannis Emiris}
\author[3]{Jan Legersk\'y}
\author[4]{Elias Tsigaridas}

\affil[1]{Department of Informatics and Telecommunications, National Kapodistrian University of Athens}

\affil[2]{ATHENA Research Center}
\affil[3]{Research Institute for Symbolic Computation,  Johannes Kepler University, Linz}
\affil[4]{Sorbonne Universit{\'e}, \textsc{CNRS}, \textsc{INRIA}, 
Laboratoire d'Informatique de Paris 6 (\textsc{LIP6}), {\'E}quipe \textsc{PolSys}}

\date{}
\begin{document}
	\maketitle
\begin{abstract}
  The number of embeddings of minimally rigid graphs in $\mathbb{R}^D$
  is (by definition) finite, modulo rigid transformations, for every
  generic choice of edge lengths.  Even though various approaches have
  been proposed to compute it, the gap between upper and lower
  bounds is still enormous.  Specific values and its asymptotic
  behavior are  major and fascinating open problems in rigidity theory.
  
  Our work considers the maximal number of real embeddings of
  minimally rigid graphs in $\mathbb{R}^3$.  We modify a commonly used
  parametric semi-algebraic formulation that exploits the
  Cayley-Menger determinant
  to minimize the {\em a priori} number of complex embeddings, where the
  parameters correspond to edge lengths.  To cope
  with the huge dimension of the parameter space and find specializations of
  the parameters that maximize the number of real embeddings, we
  introduce a method based on coupler curves
  that makes the sampling feasible for spatial minimally rigid
  graphs.
  
  Our methodology results in the first full classification of the
  number of real embeddings of graphs with 7 vertices in
  $\mathbb{R}^3$, which was the smallest open case. Building on this and
  certain 8-vertex graphs,
  we improve the previously known general lower bound on the maximum
  number of real embeddings in $\mathbb{R}^3$.
     
\end{abstract}


\section{Introduction}

Rigid graph theory is a very active area of  research with many applications 
in robotics \cite{Rob1,Rob2,Drone}, structural bioinformatics \cite{Em_Ber,Bio2},
sensor network localization \cite{sensor} and architecture~\cite{arch}. 

A graph embedding in $\mathbb{R}^D$, equipped with the standard
euclidean norm, is a function that maps the vertices of a graph $G$ to
$\mathbb{R}^D$. Let $V_G$, resp.\ $E_G$, denote the set of vertices,
resp.\ edges, of $G$.  We are interested in embeddings that are
compatible with edge lengths, namely, if two vertices are connected by
an edge, then the distance between them equals a given length for this
edge.  A graph $G$ is {\em generically rigid} if all embeddings
compatible with generic edge lengths are not continuously deformable.
If any edge removal results in a non-rigid mechanism, then the graph
is {\em minimally rigid}.  For $D=2$ these graphs are called
\emph{Laman graphs}.  For $D=3$, following~\cite{GraKouTsiLower17}, we
call these graphs \emph{Geiringer graphs} to honor Hilda
Pollaczek-Geiringer who worked on rigid graphs in $\RR^2$ and $\RR^3$
many years before Laman~\cite{Geiringer1932,Geiringer1927}.

For a graph $G$, let $r_D(G,\mathbf{d})$ be the number of
embeddings in~$\RR^D$ that are compatible with edge lengths
$\mathbf{d}=(d_e)_{e\in E_G}\in \RR_+^{|E_G|}$ modulo rigid motions,
and let $r_D(G)$ be the maximum of $r_D(G,\mathbf{d})$ over all
$\mathbf{d}$ such that $r_D(G,\mathbf{d})$ is finite.
To indicate the maximum number of real embeddings over all  graphs
with $n$ vertices, we write $r_D(n)$.
In this setting, an important question is to find all the
possible real embeddings of graphs with $k$ (some constant) number of
vertices.
This can be used to enumerate and classify
conformations of proteins, molecules \cite{Em_Ber,Bio2}
and robotic mechanisms, e.g.,  \cite{Drone,Diet}.
Furthermore, precise bounds for $r_D(G)$ or $r_D(k)$
are of great importance,
since gluing many copies of $G$ together yields lower bounds 
for $r_D(n)$, for $n \geq k$, e.g., \cite{Borcea2,GraKouTsiLower17}.

A natural approach to bound $r_D(G)$  is to use an algebraic
formulation to express the embeddings as solutions of a polynomial
system. The number of its complex solutions bounds the
number of complex embeddings, $c_D(G)$, which bounds $r_D(G)$.

For $D=2$, there is a recent algorithm  \cite{Joseph_lam} to solve the problem
of complex embeddings, $c_2(G)$, of minimally rigid graphs in~$\mathbb{C}^2$, for
any given graph $G$.  Besides this graph-specific
approach, using determinantal varieties \cite{Borcea1,Borcea2} we can
estimate asymptotic bounds, see also \cite{Emiris1,Steffens}; this
approach also gives results for $D=3$. Complex bounds for certain
cases of Laman graphs are also given in~\cite{Jackson1}.
For graphs with a constant number of vertices, we know that
${r_2(6) = 24}$, where the proof technique uses the
coupler curve of the Desargues graph~\cite{Borcea2}, and $r_2(7) = 56$, proved by
delicate stochastic methods~\cite{EM}. The second bound yields the
best known lower bound for Laman graphs, which is $r_2(n) \geq 2.3003^n$.

For $D=3$, the problem is much more difficult than in the planar case.
One of the reasons is that, unlike the planar case, we lack a
combinatorial characterization of minimally rigid (Geiringer) graphs
in $\RR^3$. The existence of such a characterization is a major open
problem in rigid graph theory.  The algebraic formulation considers
the squared distance between two points, not as a metric, but as the
sum of squares of the coordinates. Then, for every edge~$v_iv_j$, we have
the equation
\begin{equation*}
d^2_{ij}=(x_{i}-x_{j})^2+(y_{i}-y_{j})^2+(z_{i}-z_{j})^2  \,,
\end{equation*}
where $x_i,y_i,z_i$ are, in general complex, coordinates of a vertex
$v_i$ and $d_{ij}$ is the length of the edge~$v_iv_j$.
If we use the B\'ezout bound to bound the number of the complex roots
of the polynomial system, then the upper bound for $c_3(n)$ is
$\mathcal{O}(2^{3n})$, which is a very loose bound.
Hence, a more sophisticated approach is needed.
Nevertheless, this formulation has been successfully used to obtain upper
bounds of $r_3(k)$ via mixed volume computation of sparse polynomial
systems for 1-skeleta of simplicial polyhedra (a subset of spatial
rigid graphs) with $k \leq 10$ vertices \cite{Emiris1}.
The best known lower bound is 
$r_3(n) \geq 2.51984^n$ \cite{Emiris1}.
We improve it to $2.6553^{n}$.

As our goal is to estimate the number of real embeddings, we are
interested in the number of real solutions of the corresponding
polynomial systems.  If we consider the edge lengths as parameters,
then we are searching for specializations of the parameters that maximize
the number of real solutions of the system and, if possible, to match the
number of complex solutions.  However, the number of parameters is
very big even for graphs with a small number of vertices.  Even more,
it is an open question in real algebraic geometry to determine if the
number of real solutions of a given algebraic system is the same as
the number of complex ones up to its parameters.  While there are some
upper bounds for the number of real positive roots \cite{Sottile},
they are generally worse than mixed volume in the case of rigid
embeddings.  In addition, sparse polynomials have also been used to obtain lower
bounds of the number of real positive roots of polynomial systems, see~\cite{BRS-few-08,bispe}
and references therein.
Therefore, we need a delicate method to sample in an efficient way the parameter space
and maximize the number of real solutions that correspond to embeddings.

For graphs with a given number of vertices, we have a complete
classification for all graphs with $n \leq 6$ vertices.  Moreover, for
the case of the cyclohexane we know the tight bound of $r_3(6) = 16$ embeddings
\cite{Em_Ber}.
Let us also mention that for certain applied cases there are ad
hoc methods.  For example, the maximal number of real embeddings
of Stewart platforms was computed \cite{Diet} using a combination of
Newton-Raphson and the steepest descent method.

\paragraph{Our contribution}
We extend existing results about the number of
the spatial embeddings of minimally rigid graphs.  We construct all
minimally rigid graphs up to 8 vertices and we classify them
according to the last Henneberg step.  Then, we model our problem
algebraically using two different approaches.  Using the algebraic
formulation, we compute upper bounds for the number of complex
embeddings of all graphs with $7$ and $8$ vertices.  Then, we
introduce a method, inspired by coupler curves, to search efficiently
for edge lengths that increase the number of real embeddings.  We
provide an open-source implementation of our method in Python
\cite{sourceCode}, which  
uses PHCpack~\cite{phcpy} for solving polynomial systems. 
To the best of our knowledge there is no other
similar technique, let alone an open-source implementation. 
Based on
our formulation and software, we performed extensive experiments that
resulted in a complete classification and tight bounds for the real
embeddings for all 7-vertex Geiringer graphs, which was the smallest
open case.  Moreover, we extend our computations to certain 8-vertex
graphs. Even though the computations do not provide a full classification of
real embeddings, they are enough to improve the currently known lower bound
on the number of embeddings in $n$, namely $r_3(n) \geq 2.6553^{n}$.

\paragraph{Organization}
The rest of the paper is organized as follows. In Section~\ref{sec:alg-model} we present
the equations and inequalities of our modeling.
In Section~3, we introduce a method for
parametric searching for edge lengths inspired by coupler curves.
In Section~4, we present~$r_3(G)$ for all $G$ with 7 vertices and we
establish a new lower bounds on the maximum number of real
embeddings. Finally, in Section~5 we conclude and present some open
questions.

\section{Preliminaries \& Algebraic Modeling}
\label{sec:alg-model}

First, we present some general results about rigidity in $\RR^3$ and
then two algebraic formulations of the problem of graph embeddings.
The first, in Section~\ref{sec:eq-sphere}, is based on 0-dimensional
varieties of sphere equations.  The second, in
Section~\ref{sec:distance-sys}, exploits determinantal varieties of
Cayley-Menger matrices and inequalities.

\subsection{Rigidity in $\mathbb{R}^3$}
\label{sec:rigidity}
The first step is the construction of all minimally
rigid graphs up to isomorphism for a given number of vertices.
The combinatorial characterization of minimally rigid graphs in
dimension 3 is a major open problem.  It is well known
that ${\lvert E_G \rvert =3\lvert V_G \rvert-6}$, and
$\lvert E_H \rvert \leq 3\lvert V_H \rvert-6$ for every subgraph $H$
of $G$, but this condition is not sufficient for rigidity~\cite{handbook1}.

It is known that adding a new vertex to a Geiringer graph together with adding and removing certain edges
yields another Geiringer graph. These operations are called \emph{Henneberg steps}~\cite{tay}.
Henneberg step~I (H1) adds a 
vertex of degree 3, connecting it with 3 vertices in the original graph.
Henneberg step~II (H2) deletes an edge from the original graph, a new
vertex is connected to the vertices of the deleted edge and to two
other vertices of the graph, see Figure~\ref{fig:henneberg}.
For these two steps, the opposite implication also works:
If the resulting graph is Geiringer, the original one is Geiringer too.
Since the necessary condition for the number of edges guarantees that all Geiringer graphs with $\le 12$ vertices
do not have all vertices of degree greater or equal to 5, 
H1 and H2 are sufficient to construct all Geiringer graphs with $\le 12$ vertices.

There are two additional steps,
the so-called X-replacement and double V-replacement (H3x and H3v).
They extend rigid graphs in~$\mathbb{R}^3$ with a 
vertex of degree 5, see Figure~\ref{fig:henneberg}.
Every minimally rigid graph in $\mathbb{R}^3$ can be constructed by a
sequence of steps H1, H2, H3x or H3v starting from a tetrahedron.
On the other hand, it is not proven whether these moves construct only rigid graphs
\cite{handbook1} (for~dimension 4 there is a counterexample such that
4-dimensional variant of 
H3x gives a non-rigid graph \cite{H3contre}).

\begin{figure}
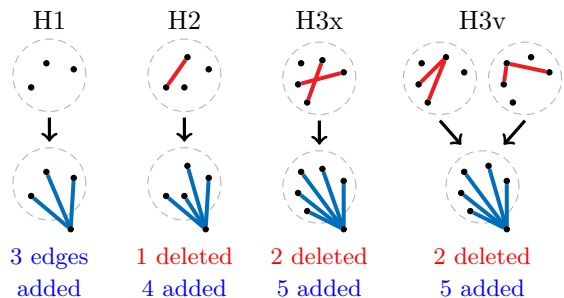
	
	\begin{center}
		\begin{tabular}{cccc}
			H1  & H2 & H3x  & H3v  \\
			\HI	& \HII  	& \HIIIx & \HIIIv \\
			\small{\Blue{3 edges}}& \small{\Red{1 deleted}} &	\small{\Red{2 deleted}} & \small{\Red{2 deleted}} \\	
			\small{\Blue{added}}& \small{\Blue{4 added}} &	\small{\Blue{5 added}} & \small{\Blue{5 added}} \\	
		\end{tabular} 
	\end{center}
	\caption{Henneberg steps in $\mathbb{R}^3$ }
	\label{fig:henneberg}
\end{figure}

\subsection{Equations of spheres}
\label{sec:eq-sphere}

\begin{defn}		\label{def:magnitudeEquations}
	If $G=(V_G,E_G)$ is a graph with edge lengths $\mathbf{d}=(d_e)_{e\in E_G}\in \RR_+^{|E_G|}$ and $v_1,v_2,v_3 \in V_G$ are such that $v_1v_2,v_2v_3,v_1v_3\in E_G$, 
	then $S(G,\mathbf{d},v_1v_2v_3)\subset (\mathbb{C}\times \mathbb{C} \times \mathbb{C})^{|V_G|}$ denotes the zero set of the following equations
	\begin{align*}
	(x_{v_1},y_{v_1},z_{v_1}) = (0,0,0), \, \,
	(x_{v_2},y_{v_2},z_{v_2}) &= (0,d_{v_1v_2},0), \\
	(x_{v_3},y_{v_3},z_{v_3}) = (x_3,y_3,0),\,\,
	x_v^2 + y_v^2 + z_v^2 &= s_v \,\,\forall v \in V_G\,, \\
	s_u +s_v -2(x_u x_v+y_u y_v+z_u z_v) &= d_{uv}^2 \,\, \forall uv \in E_G\,,
	\end{align*}
	where $x_3,y_3$ are such that $x_3\geq 0$, $x_3^2+y_3^2 = d^2_{v_1v_3}$ and $x_3^2+(y_3-d_{v_1v_2})^2 = d^2_{v_2v_3}$.
	We denote the real solutions $S(G,\mathbf{d},v_1v_2v_3)\cap(\RR\times \RR \times \RR)^{|V_G|}$ by  $S_\RR(G,\mathbf{d},v_1v_2v_3)$.
\end{defn}

The first 3 equations remove rotations and translations. 
The distances of vertices from the origin are expressed by new (nonzero) variables
to avoid roots at toric infinity which prohibit 
mixed volume from being tight \cite{Emiris1,Steffens}.
The other equations are distances between embedded points.

Notice that $r_3(G,\mathbf{d})=|S_\RR(G,\mathbf{d},v_1v_2v_3)|$.
If $\mathbf{d}$ is generic, then $c_3(G,\mathbf{d})=|S(G,\mathbf{d},v_1v_2v_3)|=c_3(G)$ since the number of complex embeddings is a generic property.
The mixed volume of the system depends on the choice of the fixed triangle.
Hence, all possible choices must be tested for some graphs in order to get the best possible bound.

\subsection{Distance geometry}
\label{sec:distance-sys}

Distance geometry is the study of the properties of points given only
the distances between them.  A basic tool is
the squared distance matrix, extended by a row and a column of ones
(except for the diagonal), known as Cayley-Menger matrix~\cite[Chapter IV, Section~40]{Blu}:
\begin{equation*}
CM=\begin{pmatrix}
0 & 1 & 1 & \cdots & 1\\ \vspace{-0.3em}
1& 0  & d^2_{12}     & \cdots        & d^2_{1n} \\ \vspace{-0.3em}
1& d^2_{12}  & 0 & \ddots  &   \ldots \\
\cdots & \cdots &  \ddots & \ddots & \ldots   \\
1& d^2_{1n}  & d^2_{2n} & \cdots & 0  
\end{pmatrix}\,,
\end{equation*}
where $d_{ij}$ is the distance between point $i$ and $j$.
The points with such distances are embeddable in $\mathbb{R}^D$ if and only if 
\begin{itemize}
	\item $\text{rank}(CM)=D+2$ and
	\item $(-1)^k \det(CM')\geq 0$, for every submatrix $CM'$ 
	with size $k+1\leq D+2$ that includes the extending row and column.
\end{itemize}
The distances among all $n$ points correspond to edge lengths of the
complete graph with $n$ vertices.  Hence, assuming that lengths of
non-edges of our graph $G$ correspond to variables, the first
condition gives rise to determinantal equations.  This condition
suffices for embeddings in~$\mathbb{C}^D$.  The systems of these
equations are overconstrained (for example $21$ equations in $6$
variables for $n=7$ and 56 equations in $10$ variables for $n=8$).
The second embedding condition can be interpreted by geometrical
constraints on the lengths.  For $k=2$ this means simply that a length
should be positive. For $k=3$ the resulting inequality is the
triangular one, while for $k=4$ we obtain \textit{tetrangular
	inequalities}.  The latter can be seen as a generalization of the
triangular ones, since they state that the area of no triangle is
bigger than the sum of the other three in a tetrahedron.

Although the systems of equations are overconstrained, a square
subsystem can be found.  The question is if these subsystems can give
us information for the whole mechanism.  In \cite{Jackson2}, the
authors present an idea relating Cayley-Menger subsystems with
\textit{globally rigid} graphs.  They are a certain class of graphs
consisting of mechanisms with unique realizations up to rigid motions
and reflections.  If extending $G$ by the edges corresponding to the
variables of the square subsystem yields a globally rigid graph, then
the number of solutions of the reduced system gives an upper bound for
the whole system.  Since the reflections are factored out by the
distance system, the number of solutions is $c_3(G)/2$.  We check
global rigidity using stress matrices derived from rigidity
matroids~\cite{Global}.

It is easy to find square subsystems from the determinantal
equations.  The question is what is the smallest number of variables
needed to establish an upper bound and if this subsystem
captures all solutions of the whole graph.
The following lemma provides an estimate of the number of variables.
\begin{lem}
	For every minimally rigid graph $G$ in dimension 3, there is at least one extended graph $H=G\cup \{e_1, e_2,..,e_k\}$,
	with $k=|V_G|-4$ and $e_i \notin E_G$, which is globally rigid in $\mathbb{C}^3$.
\end{lem} 
\begin{proof}
	The only 5-vertex minimally rigid graph is obtained by applying an H1 step to the tetrahedron.  
	If we extend this graph with the only non-existing edge,
	we obtain the complete graph in 5 vertices, so the lemma holds. 
	Let the lemma hold for all graphs with $n$ or less vertices. 
	For every graph obtained by an H2 step, the lemma holds since H2 preserves global rigidity \cite{Connelly}. 	
	
	Let a graph $G_{n+1}$ be constructed by an H1 move applied to an $n$-vertex graph $G_n$, whose extended globally rigid graph is $H_n$. 
	Without loss of generality, this move connects a new vertex $v_{n+1}$ with the vertices $v_{1},v_{2},v_{3}$. 
	Let $u$ be a neighbour of $v_1$ in $G_{n+1}$ such that $v_2\neq u\neq v_3$. 
	The edge $uv_1$ exists also in $G_{n}$ and $H_n$. 
	If we set $H'_{n+1}=(H_{n}\cup \{v_1v_{n+1}, v_2v_{n+1}, v_3v_{n+1}, uv_{n+1}\})-\{v_1u\}$, then $H'_{n+1}$ is globally rigid, 
	because it is constructed from $H_n$ by an H2 step.
	Hence, $H_{n+1}=H'_{n+1}\cup  \{ uv_{n+1}\}$ is also globally rigid, proving the statement in the case of H1 steps.
	
	As for H3 steps, both H3x and H3v can be seen as an H2 step followed by an edge deletion. Extending the graph with the second deleted edge preserves global rigidity.
\end{proof}

\begin{figure}[!b]
	\begin{center}
		\begin{tikzpicture}[yscale=0.85, xscale=0.85]
		\begin{scope}		
		\maxSevenVert
		\draw[edge] (2)edge(3)  (3)edge(4) (6)edge(2) (5)edge(4) (6)edge(5);
		\draw[edge] (2)edge(1) (1)edge(4) (1)edge(3) (1)edge(5) (1)edge(6);
		\draw[edge] (2)edge(7) (7)edge(4) (7)edge(3) (7)edge(5) (7)edge(6);
		\draw[rdedge] (2)edge(4) (4)edge(6) (1)edge(7);
		\node[vertex] at (1) {};
		\node[vertex] at (2) {};
		\node[vertex] at (3) {};
		\node[vertex] at (4) {};
		\node[vertex] at (5) {};
		\node[vertex] at (6) {};
		\node[vertex] at (7) {};
		\node[below right=0.08cm] at (1) {$v_1$};
		\node[above left=0.06cm] at (2) {$v_2$};
		\node[below left=0.11cm] at (3) {$v_3$};
		\node[below right=0.13cm] at (4) {$v_4$};
		\node[above right=0.06cm] at (5) {$v_5$};
		\node[above left=0.01cm] at (-0.3, 0.18) {$v_6$};
		\node[above right=0.06cm] at (0,1) {$v_7$};
		\begin{scope}[xshift=5.1cm]
		\maxSevenVert
		\draw[edge] (2)edge(3)  (3)edge(4) (6)edge(2) (5)edge(4) (6)edge(5);
		\draw[edge] (2)edge(1) (1)edge(4) (1)edge(3) (1)edge(5) (1)edge(6);
		\draw[edge] (2)edge(7) (7)edge(4) (7)edge(3) (7)edge(5) (7)edge(6);
		\draw[rdedge] (1)edge(7);
		\node[vertex] at (1) {};
		\node[vertex] at (2) {};
		\node[vertex] at (3) {};
		\node[vertex] at (4) {};
		\node[vertex] at (5) {};
		\node[vertex] at (6) {};
		\node[vertex] at (7) {};
		\node[below right=0.08cm] at (1) {$v_1$};
		\node[above left=0.06cm] at (2) {$v_2$};
		\node[below left=0.11cm] at (3) {$v_3$};
		\node[below right=0.13cm] at (4) {$v_4$};
		\node[above right=0.06cm] at (5) {$v_5$};
		\node[above left=0.01cm] at (-0.3, 0.18) {$v_6$};
		\node[above right=0.06cm] at (7) {$v_7$};
		\node[right] at (0.0,-0.07) {\textcolor{red}{$x_1$}};
		\end{scope}
		\end{scope}
		\end{tikzpicture}
	\end{center}	
	\caption{The graph $G_{48}$ (grey edges). There are submatrices of $CM_{G_{48}}$ that involve only variables corresponding to the 3 red dashed edges of the left graph.
		The graph $G_{48}$ extended by the edge $v_1v_7$ (that corresponds to the variable $x_1$) is globally rigid. 
	}
	\label{fig:G48}
\end{figure}
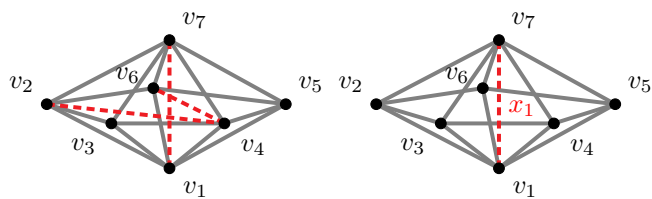

We can extend this result to minimally rigid graphs in arbitrary dimension constructed by appropriate
generalizations of Henneberg steps H1, H2 or H3.  As we mentioned, the lemma gives only an estimate
for the smallest number of variables.  It guarantees neither that
such subsystem exists in every Cayley-Menger matrix of a minimally
rigid graph (in fact we have found graphs with 8 or more vertices with
no such a subsystem), nor that the solutions of the subsystem
totally define the whole system.  On the other hand, if such a subsystem
exists, it can definitely give an upper bound.

An example is the 7-vertex graph
$G_{48}$ with the maximal number of embeddings ($r_3(G_{48})=48=r_3(7)$, see
Section~\ref{sec:results}).  The 
labeling of the vertices is in Figure~\ref{fig:G48}.
There are 5 different square systems in 
$3$ variables that completely define the mechanism.  We can choose one
of them involving only $x_1, x_2, x_3$:
\begin{equation*}
CM_{G_{48}}=\begin{pmatrix} 
0&1&1&1&1&1&1&1  \\ 
1 & 0 & d^2_{12}& d^2_{13}& d^2_{14}&d^2_{15} & d^2_{16} & \Red{x_{1}}		\\ 
1 & d^2_{21} & 0 & d^2_{23} & \Red{x_{2}} & \Red{x_{3}} & d^2_{26} & d^2_{27} \\ 
1& d^2_{31} & d^2_{32} & 0 & d^2_{34} & \Blue{x_{4}} & \Blue{x_{5}}	& d^2_{37} \\ 
1& d^2_{41} & \Red{x_{2}}& d^2_{43} & 0& d^2_{45} & \Blue{x_{6}} & d^2_{47}\\ 
1& d^2_{51} & \Red{x_{3}} & \Blue{x_{4}} & d^2_{54}& 0& d^2_{56}& d^2_{57}		\\
1& d^2_{61}& d^2_{62} & \Blue{x_{5}} & \Blue{x_{6}} & d^2_{65}& 0& d^2_{67}		\\
1&\Red{x_{1}}& d^2_{72} & d^2_{73} & d^2_{74} & d^2_{75} & d^2_{76}& 0
\end{pmatrix}\,.
\end{equation*}
One advantage of this approach is that we have much less equations
compared with the sphere equations approach. In this example, we need
a system of only 3 equations for the distance system, while 16
equations are required otherwise.
Additionally, every solution of the
distance system corresponds to two reflected embeddings.
Hence, polynomial homotopy solvers are much faster in this case.

We can also apply algebraic elimination to reformulate this determinantal variety.
We noticed that even the graph extended only with the edge $v_1v_7$ corresponding to the variables $x_1$ is globally rigid. 
This led us to compute the resultant of the square 3x3 system for $x_1$,
which can be obtained by repeated Sylvester resultants, Macaulay resultant and sparse resultant method with the same result.
In order to specify the realizations, we also need the set of inequalities. There are $35$ triangular inequalities and the same number of tetrangular inequalities for  the whole set of variables.
Since we need to embed only one new edge, we are restricted to find the inequalities for $x_1$.
There are ten inequalities that include only $x_1$ (5 triangular and 5 tetrangular).

On the other hand, we detected graphs for which the subsystems do not fully describe the determinantal variety, since the number of solutions of the whole (overconstrained) system is smaller than this of the (square) subsystem for some generic choices of lengths.  We conclude that the drawback of the method is that there is not a 1-1 correspondence between subsystems and global rigidity.
Despite this fact, they seem better candidates for tight upper bound mixed volume computations.

\section{Increasing the number of real embeddings}
\label{sec:coupler}

To improve $r_3(G)$ bounds, our first goal was to prove that
$r_3(G_{48})=c_3(G_{48})$. Initially we used methods already applied
to increase the number of real solutions of a given polynomial
system. We present a short overview of this approach.
\paragraph{Stochastic methods} 
A first idea was to use stochastic sampling. Generic configurations of~$G_{48}$ embeddings in $\mathbb{R}^3$ were perturbed following the sampling methods presented in \cite{EM}.
Applying this approach, it was straightforward to find configurations with $r_3(G_{48},\mathbf{d})$ being equal to $16,20$ or $24$.
Our best result was $r_3(G_{48},\mathbf{d})=32$.  
\paragraph{Parametric searching with CAD method}
\label{sec:parametric} 
\texttt{Maple}'s subpackage \texttt{RootFinding [Parametric]}
implements Cylindrical Algebraic Decomposition principles for
semi-algebraic sets \cite{parametric}.  This implementation could not work for the
system of sphere equations, but was efficient using the semi-algebraic
distance system.  The algorithm can separate variables and parameters
for every equation and give as output a decomposition of the space
of parameters up to the number of solutions.  In our case, it was
possible to use only one parameter due to computational constraints,
so all the other distances were fixed
(our Maple worksheet is available at \cite{sourceCode}).

It was again straightforward to find $24$ embeddings even from totally
random conformations.  To get more we needed to exploit the
symmetry of $G_{48}$, constructing non-generic flexible frameworks.
Perturbing the lengths by a small quantity, $r_3(G_{48},\mathbf{d})$ was
again finite.
Afterwards, we considered multiple edge lengths as
linear combinations of the same parameter.  Eventually, applying
parameter searching, we were able to find lengths $\mathbf{\bar{d}}$
such that $r_3(G_{48},\mathbf{\bar{d}})=28$:\par\nobreak {\parskip0pt
	\footnotesize \noindent
	\begin{equation}\label{eq:edgelengths28}
	\begin{aligned}
	\bar{d}_{12} &=  1.99993774567597 , & \bar{d}_{27} &=  10.5360917228793 , \\
	\bar{d}_{13} &=  1.99476987780024 , & \bar{d}_{37} &=  10.5363171636461 , \\
	\bar{d}_{14} &=  2.00343646098439 , & \bar{d}_{47} &=  10.5357233031495 , \\
	\bar{d}_{15} &=  2.00289249524296 , & \bar{d}_{57} &=  10.5362736599978 , \\
	\bar{d}_{16} &=  2.00013424746814 , & \bar{d}_{67} &=  10.5364788463527 , \\
	\bar{d}_{23} &=  0.99961432208948, & \bar{d}_{34} &=  1.00368644488060 ,\\
	\bar{d}_{45} &=  1.00153014850485 , & \bar{d}_{56} &=  0.99572361653574 ,\\ 
	\bar{d}_{26} &=  1.00198771097407
	\end{aligned}
	\end{equation}
}

While this result was lower than the one achieved by stochastic searching,
it had some promising properties (variables are taken from the determinantal  variety of $CM_{G_{48}}$).
Namely, all the solutions for $x_1,x_2,x_3$ are real, for $x_1$ even positive, and the $x_1$ solutions which are not embeddable are very close to the intervals imposed by the triangular and tetrangular inequalities. 

\paragraph{Gradient Descent}
An algorithm that increases step by step the number of real embeddings is proposed in \cite{Diet}.
This method is based on gradient descent optimization, minimizing the imaginary part  of solutions, while forcing existing real roots to remain real via a semidefinite relation. 

We applied it to the $G_{48}$ sphere equations and a variant of it for distance equations starting from the optimal configurations found with the two previous approaches. 
In the first iterations the results were encouraging, but finally we could not generate more real embeddings.

The previous results motivated us to search other ways to achieve our
first goal.  Inspired by coupler curve visualization, we introduce an
iterative procedure that modifies edge lengths so that the number of
real embeddings might increase.  In particular, it allows to find edge
lengths to prove that $r_3(G)=c_3(G)$ for $G_{48}$ and also other
7-vertex graphs $G$.  At each iteration, only lengths of 4 edges in a
specific subgraph are changed.  One can be changed freely, whereas the
other 3 are related.  For this two-parametric family, we search values
with the maximal number of embeddings globally.

\subsection{Coupler curve}
For a minimally rigid graph $G$, removing an edge $uc$ yields a framework $H=(V_G,E_G\setminus {uc})$ with one degree of freedom.
If we fix a triangle containing $u$ in order to avoid rotations and translations of~$H$, then the vertex $c$ draws the so called \emph{coupler curve} under all possible motions of $H$.
This idea was already used in~\cite{Borcea2} for obtaining 24 real embeddings of Desargues (3-prism) graph in~$\mathbb{R}^2$.
A modification into $\mathbb{R}^3$ is straightforward 
-- the number of embeddings of $G$ corresponds to the number of intersection of the coupler curve of $c$ of the graph $H$ with a sphere centered at $u$ with radius~$d_{uc}$.
The following definition recalls the concept of coupler curve more precisely.

\begin{defn}		\label{def:couplerCurve}
	Let $H$ be a graph with edge lengths $\mathbf{d}=(d_e)_{e\in E_H}$ and $v_1,v_2,v_3 \in V_H$ be such that $v_1v_2,v_2v_3,v_1v_3\in E_H$. 
	If the set~$S_\RR(H,\mathbf{d},v_1v_2v_3)$ is one dimensional and $c\in V_H$, 
	then the set 
	\begin{equation*}
	\mathcal{C}_{c,\mathbf{d}}=\{(x_c,y_c,z_c) \colon ((x_v,y_v,z_v))_{v\in V_H} \in S_\RR(H,\mathbf{d},v_1v_2v_3)\}	
	\end{equation*}
	is called a \emph{coupler curve of $c$  w.r.t.\ the fixed triangle $v_1v_2v_3$}.
\end{defn}

Obviously, for given lengths $\mathbf{d}$ of the graph $H$, we may vary the length $d_{uc}$ of the removed edge $uc$ 
so that the number of intersections of the coupler curve $\mathcal{C}_{c,\mathbf{d}}$ with the sphere centered at $u$ with radius $d_{uc}$,
i.e., the number of embeddings of $G$, is maximal.
The following lemma enables us to move also the center of the sphere within a certain one-parameter family without changing the coupler curve.

\begin{lem}	\label{lem:couplerCurvePreserves}
	Let $G$ be a minimally rigid graph and let $u,v,w,p,c$ be vertices of $G$ such that $pv,vw \in E$ and the neighbours of $u$ in $G$ are $v,w,p$ and $c$.
	Let $\mathcal{C}_{c,\mathbf{d}}$ be the coupler curve of $c$ of the graph $H=(V_G,E_G\setminus\{uc\})$
	with edge lengths $\mathbf{d}=(d_e)_{e\in E_H}$ w.r.t.\ the fixed triangle $vuw$.
	Let $z_p$ be the altitude of $p$ in the triangle $uvp$ with lengths given by $\mathbf{d}$.
	The set $\{y_p \colon ((x_{v'},y_{v'},z_{v'}))_{v'\in V_H} \in S_\RR(H,\mathbf{d},vuw)\}$ has only one element $y'_p$.
	If the parametric edge lengths $\mathbf{d}'(t)$ are given by 	
	\begin{align*}
	d'_{uw}(t)&=||(x_w,y_w-t,0)||\,, \quad d'_{up}(t)=||(0,y'_p-t,z_p)||\,,	\\
	d'_{uv}(t)&= t\,,  \,\text{ and }  
	d'_{e}(t)=d_{e} \text{ for all } e\in E_H\setminus\{uv,uw,up\}\,,
	\end{align*}
	then the coupler curve $\mathcal{C}_{c,\mathbf{d'}(t)}$ of $c$ w.r.t.\ the fixed triangle $vuw$ is the same for all $t\in\mathbb{R}_+$,
	namely, it is $\mathcal{C}_{c,\mathbf{d}}$. Moreover, if $cw\in E_G$, then $\mathcal{C}_{c,\mathbf{d'}(t)}$ is a spherical curve.
\end{lem}
\begin{proof}
	Within this proof, all coupler curves are w.r.t.\ the triangle $vuw$.
	The situation is illustrated by Figure~\ref{fig:couplerCurvePreserved}. 
	Since $G$ is minimally rigid, removing the edge $uc$ yields a graph $H$ such that $S_\RR(H,\mathbf{d},vuw)$ is one dimensional for a generic choice of $\mathbf{d}$.
	The set $\{y_p \colon ((x_{v'},y_{v'},z_{v'}))_{v'\in V_H} \in S_\RR(H,\mathbf{d},vuw)\}$ has indeed only one element, 
	since the coupler curve $\mathcal{C}_{p,\mathbf{d}}$ of $p$ is a circle whose axis of symmetry is the $y$-axis. 
	The parametric edge lengths  $\mathbf{d}'(t)$ are such that the position of $v$ and $w$ is the same for all $t$.
	Moreover, the coupler curve $\mathcal{C}_{p,\mathbf{d'}(t)}$ of $p$ is independent on $t$. 
	Hence, the coupler curve $\mathcal{C}_{c,\mathbf{d'}(t)}$ is independent on $t$, because the only vertices adjacent to $u$ in $H$ are $p,v$ and $w$,
	i.e., the position of $u$ does not influence positions of the other vertices.
\end{proof}

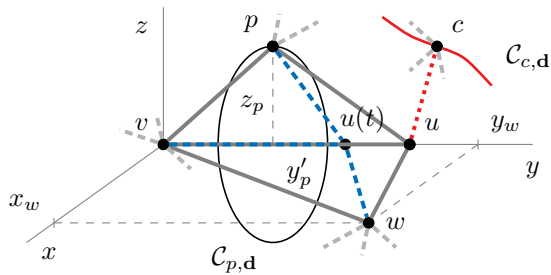
\begin{figure}[!htb]
	\begin{center}
		\begin{tikzpicture}[xscale=1.8,yscale=1.3]
		\draw[gray] (0,0) to (0,1.4) {};
		\draw[gray] (0,0) to (2.7,0) {};
		\draw[gray] (0,0) to (-1,-1) {};
		
		\node[below right=0.1cm] at (-1.0,-0.9) {$x$};
		\node[below=0.1cm] at (2.7,0) {$y$};
		\node[below left=0.1cm] at (0,1.4) {$z$};
		
		\node[vertex] (v) at (0,0) {};
		\node[vertex] (u) at (1.8,0) {};
		\node[vertex] (w) at (1.5,-0.8) {};
		\node[vertex] (p) at (0.8,1) {};
		\node[vertex] (ut) at (1.33,0) {};
		\node[vertex] (c) at (2,1) {};
		\coordinate (s) at (0.8,0);
		
		\draw[black,line width=0.6pt] (s) ellipse (0.4cm and 1cm);		
		
		\draw[gray, dashed]	(s) to (p);	
		\node[below left=0.1cm] at 	(0.8,-0.9) {$\mathcal{C}_{p,\mathbf{d}}$};
		
		\draw[gray, dashed]	(2.3,0) to (w);
		\draw[gray, dashed]	(-0.8,-0.8) to (w);
		\draw[gray] (2.3,-0.05) to (2.3,0.05);
		\draw[gray] (-0.8,-0.85) to (-0.8,-0.75);
		\node[above left=0.06cm] at (-0.8,-0.8) {$x_w$};
		\node[above right=0.06cm] at (2.3,0) {$y_w$};
		
		\draw[edge] (p) to (v);
		\draw[edge] (v) to (u);
		\draw[edge] (v) to (w);		
		\draw[edge] (u) to (w);
		\draw[edge] (p) to (u);	
		\draw[bdedge] (ut) to (w);
		\draw[bdedge] (p) to (ut);
		\draw[bdedge] (v) to (ut);
		\draw[rdedge, dotted] (u) to (c);	
		
		\draw[line width=1.0pt,Red] (c) .. controls +(-0.2,0.1) .. (1.6,1.3);	
		\draw[line width=1.0pt,Red] (c) .. controls +(0.2,-0.1) .. (2.4,0.6);			
		\node[above right=0.1cm] at (2.4,0.6) {$\mathcal{C}_{c,\mathbf{d}}$};

		\node[vertex] at (u) {};
		\node[vertex] at (p) {};		
		
		\node[above left=0.1cm] at (v) {$v$};
		\node[above right=0.1cm] at (u) {$u$};
		\node[above right=0.05cm] at (ut) {$\!\!\!\!u(t)$};
		\node[right=0.1cm] at (w) {$w$};
		\node[above left=0.1cm] at (p) {$p$};
		\node[below right=0.08cm] at (s) {$y'_p$};
		\node[above right=0.1cm] at (c) {$c$};
		\node at (0.65,0.4) {$z_p$};
		\draw[edgeq] (p) -- +(0.33,0.25);
		\draw[edgeq] (p) -- +(0.05,0.33);
		\draw[edgeq] (w) -- +(-0.33,-0.25);
		\draw[edgeq] (w) -- +(-0.05,-0.33);
		\draw[edgeq] (w) -- +(0.2,-0.25);
		\draw[edgeq] (c) -- +(-0.2,0.25);
		\draw[edgeq] (c) -- +(0.05,-0.33);
		\draw[edgeq] (c) -- +(-0.2,-0.25);
		\draw[edgeq] (v) -- +(-0.32,0.15);
		\draw[edgeq] (v) -- +(-0.05,0.3);
		\draw[edgeq] (v) -- +(0.2,-0.25);
		\end{tikzpicture}
	\end{center}
	\caption{Since the lengths of $up$ and $uw$ are changed accordingly to the length of $uv$ (blued dashed edges),
		the coupler curves  $\mathcal{C}_{p,\mathbf{d'}(t)}$ and  $\mathcal{C}_{c,\mathbf{d'}(t)}$ are independent on $t$.
		The red dashed edge  $uc$ is removed from $G$.}
	\label{fig:couplerCurvePreserved}
\end{figure}

Thus, for every subgraph of $G$ given by vertices $u,v,w,p,c$ such that $pv,vw \in E$ and the neighbours of $u$ in $G$ are $v,w,p$ and $c$,
we have a two-parametric family of lengths $\mathbf{d}(t,r)$
such that the coupler curve $\mathcal{C}_{c,\mathbf{d}(t,r)}$ w.r.t.\ the fixed triangle $vuw$ is the same for all $t$ and $r$,
where the parameter $t$ determines lengths of $uv,uw$ and~$up$, and the parameter $r$ represents the length of $uc$. 
Within this family, we look for values of $t$ and $r$ that maximize the number of embeddings.

We illustrate the method on the example of $G_{48}$. Let $\mathbf{\bar{d}}$ be edge lengths given by~\eqref{eq:edgelengths28}.
We developed a program \cite{sourceCode} that plots (using \verb+Matplotlib+~\cite{Matplotlib}) the coupler curve of the vertex $v_6$ of $G$ with the edge $v_2v_6$ removed w.r.t.\ the fixed triangle $v_1v_2v_3$.
Figure~\ref{fig:couplercurve} is created by this program.
There are 28 embeddings for $\mathbf{\bar{d}}$, but we can find position and radius of the sphere corresponding to the removed edge $v_2v_6$ such that there are 32 embeddings
by using Lemma~\ref{lem:couplerCurvePreserves} for the subgraph $(u,v,w,p,c)=(v_2, v_3, v_1, v_7, v_6)$.
This is obtained by setting 
\begin{equation}
\label{eq:lenghtsWith32}
d_{12}=4.0534,\,\, d_{27}=11.1069,\,\, d_{26}=3.8545,\,\, d_{23}=4.0519\,.
\end{equation}

\begin{figure}
	\begin{center}
		\includegraphics[width=0.31\textwidth]{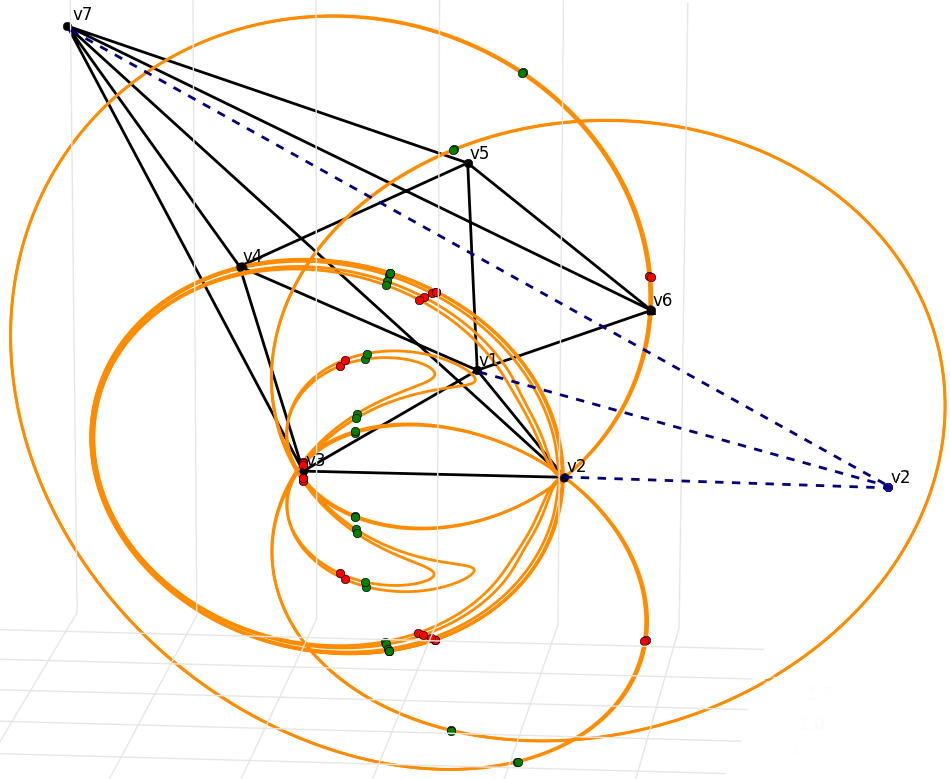}
	\end{center}
	\caption{Coupler curve $\mathcal{C}_{v_6,\mathbf{\bar{d}}}$ of $G_{48}$ with the edge $v_2v_6$ removed.
		The 28 red points are intersections of $\mathcal{C}_{v_6,\mathbf{\bar{d}}}$ with the sphere centered at $v_2$ with edge lengths $\mathbf{\bar{d}}$,
		whereas the 32 green ones are for edge lengths given by equation~\eqref{eq:lenghtsWith32} (illustrated by blue dashed lines). }
	\label{fig:couplercurve}
\end{figure}

\subsection{Sampling}
Although edge lengths of $G_{48}$ with 48 real embeddings can be
obtained by manual application of
Lemma~\ref{lem:couplerCurvePreserves} based on plots of coupler
curves, we also implemented a program~\cite{sourceCode} that searches
for a good position and radius of the sphere by sampling the parameters. 
The method and its implementation work also for minimally rigid graphs other than $G_{48}$.

We assume that the edge $cw$ is present for a suitable subgraph
(actually, this is the case for all suitable subgraphs of $G_{48}$). 
Thus, the coupler curve is spherical and 
the intersections of the coupler curve with the sphere representing the removed edge $uc$ lies on the intersection of these two spheres,
which is a circle.
Hence, instead of sampling $t$ and $r$, we sample circles on the sphere containing the coupler curve.

Since the sphere of the coupler curve is centered at $w$ and the intersecting sphere has center at~$u$,
the center of the intersection circle is on the line $uw$ and the plane of the circle is perpendicular to this line.
Hence, the circle is determined by the angle $\varphi\in(-\pi/2,\pi/2)$ between the altitude of $w$ in the triangle $uvw$ and the line $uw$,
and by the angle $\theta\in(0,\pi)$ between $uw$ and $cw$, see Figure~\ref{fig:phiTheta}.
Thus, we sample $\varphi$ and $\theta$ in their intervals, compute $t$ and $r$ from their values and select edge lengths with the highest number of real embeddings.
The algebraic systems are solved  by polynomial homotopy continuation using the Python package \verb+phcpy+ \cite{phcpy}.
In \verb+phcpy+, one can specify a starting system with the set of its solutions
instead of letting the program to construct it. 
Since the parameters change only slightly during the sampling, 
tracking the solutions of a new system from the solutions of the previous one is significantly faster 
than solving from scratch.

\begin{figure}[!htb]
	\begin{center}
		\begin{tikzpicture}[scale=1.2]
		\clip (-1.5,-1.4) rectangle (4.0,1.3);
		\draw[gray] (0,0) to (0,1.3) {};
		\draw[gray] (0,0) to (3,0) {};
		\draw[gray] (0,0) to (-1.1,-1.1) {};
		
		\node[left=0.1cm] at (-1,-1) {$x$};
		\node[below right=0.1cm] at (3,0) {$y$};
		\node[below left=0.1cm] at (0,1.3) {$z$};
		
		\node[vertex] (v) at (0,0) {};
		\node[vertex] (u) at (2.5,0) {};
		\node[vertex] (w) at (0.5,-0.8) {};
		
		\node[vertex] (c) at (1.92,-0.68) {};
		
		\coordinate (f) at (1.253,0);
		\draw[black]	(f) to (w);
		\draw[gray, dashed]	(-0.8,-0.8) to (w);
		\draw[gray] (1.3,-0.05) to (1.3,0.05);
		\draw[gray] (-0.8,-0.85) to (-0.8,-0.75);											
		
		\draw[edge] (v) to (u);
		\draw[edge] (v) to (w);		
		\draw[edge] (u) to (w);
		\draw[edge] (c) to (w);
		
		\draw[rdedge] (u) to (c);

		\node[above left=0.1cm] at (v) {$v$};
		\node[above right=0.1cm] at (u) {$u$};
		\node[below left=0.1cm] at (w) {$w$};
		\node[right=0.13cm] at (c) {$c$};
		
		\shade[ball color = gray!30, opacity = 0.3] (w) circle (1.5cm);
		\shade[ball color = red!30, opacity = 0.3] (u) circle (1.2cm);
		\tkzInterCC[R](w,1.5 cm)(u,1.2 cm) \tkzGetPoints{A}{B} 		
		
		\tkzDefPoint(3,0.2){Od}
		\tkzDrawArc[color=blue, dashed,line width=0.8pt](Od,A)(B)
		\tkzDefPoint(-0.5,-1.2){Of}
		\tkzDrawArc[color=blue,line width=0.8pt](Of,B)(A)
		
		\tkzDrawArc[color=black](u,B)(A)
		\tkzDrawArc[color=black,dashed](u,A)(B)
		\tkzDrawArc[color=black,dashed](w,B)(A)
		\tkzDrawArc[color=black](w,A)(B)
		
		\coordinate (fc) at ($(A)!0.5!(B)$);
		\draw (fc) to (c);
		
		\node[vertex] at (c) {};
		
		\tkzMarkAngle[color=black,size=0.3](v,f,w)
		\tkzLabelAngle[color=black,pos=-0.2](w,f,v){$\cdot$}
		
		\tkzMarkAngle[color=black,size=0.15](w,fc,c)
		\tkzLabelAngle[color=black,pos=0.1](w,fc,c){$\cdot$}

		\tkzMarkAngle[size=0.5,line width=0.7pt](u,w,f)
		\tkzLabelAngle[pos=0.6](u,w,f){$\varphi$}		
		\tkzMarkAngle[size=0.7,line width=0.7pt](c,w,u)
		\tkzLabelAngle[pos=0.8](c,w,u){$\theta$}
		\end{tikzpicture}
	\end{center}
	\caption{For fixed position of $v$ and $w$, the angle $\varphi$ determines the position of $u$, since $u$ lies on the $y$-axis. 
		If also the length of $cw$ is given, then $\theta$ determines the length of $uc$. The intersection circle is blue.}
	\label{fig:phiTheta}
\end{figure}
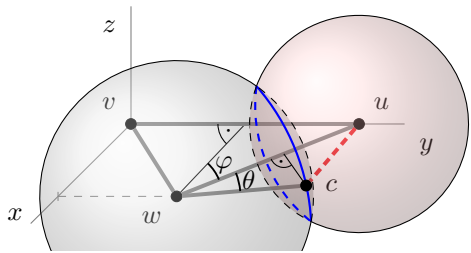

\subsection{More subgraphs suitable for sampling}
Usually, one iteration of the sampling produces many edge lengths with the same number of real embeddings.
If this number is not the desired one, then we need to pick starting edge lengths for the next iteration with a different subgraph suitable for sampling.
Our heuristic choice is based on clustering of pairs $(\varphi,\theta)$ using the function \verb+DBSCAN+ from the \verb+sklearn+ Python package~\cite{sklearn}.
From each cluster, we pick the center of gravity as $(\varphi,\theta)$ for the output lengths,
or the pair $(\varphi,\theta)$ closest to this center if the edge lengths corresponding to the center have less real embeddings.

We tested two approaches in sampling for subgraphs:
\begin{enumerate}
	\item \emph{Tree search} -- we apply the procedure using all suitable subgraphs for given starting lengths 
	and then we do the same recursively for all outputs whose number of embeddings increased until the required number is reached (or there are no increments).
	We trace the state tree depth-first.
	\item \emph{Linear search} -- we order all suitable subgraphs and an output from the procedure applied to starting lengths with the first subgraph
	is used as input for the procedure with the second subgraph, etc. 
	There is also branching because of multiple clusters -- we test all of them in depth-first way.
\end{enumerate}

\section{Classification and Lower Bounds}
\label{sec:results}

Henneberg steps may result in isomorphic graphs either
constructed by the same H-step or by another one.  We recall that no
H3x or H3v step is needed for 7 and 8-vertex graphs.  We classify
each graph up to isomorphism by the sequence of Henneberg steps
needed for its construction.  We use a certain hierarchy for this
classification: on the one hand there are graphs that can be constructed
by an H1 move in the last step, while for the others H2 is needed.
This process is important, since H1 steps trivially
double the number of real embeddings as the new vertex lies in the
intersection of 3 spheres.  This means that the number of embeddings
for H1 graphs is already known, assuming that we know the number
for the parent graph. Our \texttt{MATLAB} and \texttt{SageMath}
implementations, which verify each other,
were used to apply Henneberg steps and remove isomorphisms (see \cite{sourceCode}).
This is not a computationally difficult task for $n=7$ or $8$.
We remark that this is also done in \cite{GraKouTsiLower17} up to $10$ vertices.

The first estimate of $c_3(G)$ is the mixed volume of the algebraic
systems.  Let $f$ be a square polynomial system in $m$ variables.  The
convex hull of the exponents vector of each polynomial is its Newton
polytope.  The mixed volume of the polytope bounds
the number of solutions and is tight generically
in $(\mathbb{C}^{*})^m$.  We computed the mixed
volume for both sphere and distance equations.  We solved
the systems for random edge lengths and checked whether the mixed volume
bound was tight in all cases.  Finally, we used the method
in Section~\ref{sec:coupler} to find parameters 
maximizing the number of real embeddings.  

\subsection{7-vertex graphs}
For $n=6$, there are three H1 graphs and
one obtained with an H2 step -- the cyclohexane $G_{16}$.
The number of real embeddings of the H1 graphs is $8$, while it is known that $r_3(G_{16})=16$ \cite{Em_Ber}. 
One can also obtain lengths $\mathbf{d}$ such that $r_3(G_{16},\mathbf{d})=16$ by our method within a few tries with random starting lengths.
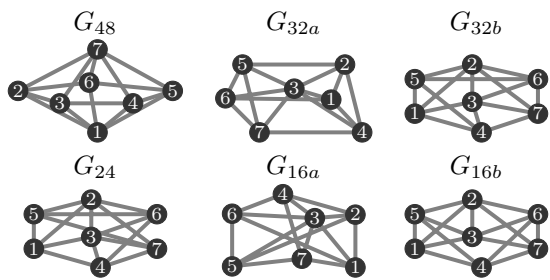
\begin{figure}[!htb]
	\begin{center}
		\begin{tabular}{ccc}
			$G_{48}$ & $G_{32a}$ & $G_{32b}$ \\
			\begin{tikzpicture}[scale=0.55]
			\coordinate(1) at (0, -1);
			\coordinate (2) at (-1.9, 0);
			\coordinate (3) at (-0.9, -0.3) ;
			\coordinate (4) at (0.85, -0.3) ;
			\coordinate (5) at (1.8,0.0) ;
			\coordinate (6) at (-0.2, 0.2) ;
			\coordinate (7) at (0,1) ;
			\draw[edge] (2)edge(3)  (3)edge(4) (6)edge(2) (5)edge(4) (6)edge(5);
			\draw[edge] (2)edge(1) (1)edge(4) (1)edge(3) (1)edge(5) (1)edge(6);
			\draw[edge] (2)edge(7) (7)edge(4) (7)edge(3) (7)edge(5) (7)edge(6);
			\node[lnode] at (1) {1};
			\node[lnode] at (2) {2};
			\node[lnode] at (3) {3};
			\node[lnode] at (4) {4};
			\node[lnode] at (5) {5};
			\node[lnode] at (6) {6};
			\node[lnode] at (7) {7};
			\end{tikzpicture}
			& 
			\begin{tikzpicture}[scale=0.9]
			\coordinate (1) at (2.073556820291433,0.55260245625217);
			\coordinate (2) at (2.27496378692584,1.069420195538877);
			\coordinate (3) at (1.53393815496906,0.71405847586394);
			\coordinate (4) at (2.525772462357367,0.069985890888852);
			\coordinate (5) at (0.777711997228549,1.069420195538877);
			\coordinate (6) at (0.526903321797022,0.57160310843183);
			\coordinate (7) at (1.024720541214144,0.069985890888852);
			\draw[edge] (4) to (7)  (1) to (3)  (5) to (6)  (1) to (4)  (2) to (3)  (3) to (7)  (2) to (5)  (3) to (5);
			\draw[edge] (1) to (2)  (6) to (7)  (5) to (7)  (3) to (6)  (1) to (6)  (3) to (4)  (2) to (4) ;
			\node[lnode] at (1) {1};
			\node[lnode] at (2) {2};
			\node[lnode] at (3) {3};
			\node[lnode] at (4) {4};
			\node[lnode] at (5) {5};
			\node[lnode] at (6) {6};
			\node[lnode] at (7) {7};
			\end{tikzpicture} &	
			\begin{tikzpicture}[scale=0.9]
			\cubeCoord
			\draw[edge] (1) to (2)  (4) to (7)  (2) to (6)  (4) to (5)  (1) to (4)  (5) to (6)  (1) to (3)  (2) to (3);
			\draw[edge] (3) to (7)  (2) to (5)  (2) to (7)  (6) to (7)  (1) to (5)  (3) to (6)  (3) to (4);
			\node[lnode] at (1) {1};
			\node[lnode] at (2) {2};
			\node[lnode] at (3) {3};
			\node[lnode] at (4) {4};
			\node[lnode] at (5) {5};
			\node[lnode] at (6) {6};
			\node[lnode] at (7) {7};
			\end{tikzpicture}\\
			$G_{24}$ &  $G_{16a}$ &$G_{16b}$ \\
			\begin{tikzpicture}[scale=0.9]
			\cubeCoord
			\draw[edge] (1) to (2)  (4) to (7)  (2) to (6)  (5) to (6)  (1) to (4)  (1) to (3)  (2) to (3)  (3) to (7);
			\draw[edge] (2) to (5)  (2) to (7)  (4) to (6)  (5) to (7)  (1) to (5)  (3) to (6)  (3) to (4) ;
			\node[lnode] at (1) {1};
			\node[lnode] at (2) {2};
			\node[lnode] at (3) {3};
			\node[lnode] at (4) {4};
			\node[lnode] at (5) {5};
			\node[lnode] at (6) {6};
			\node[lnode] at (7) {7};
			\end{tikzpicture}& 
			\begin{tikzpicture}[xscale=1.7,yscale=1.4]
			\coordinate (1) at (1.2140625924293047,0.2672685997160432);
			\coordinate (2) at (1.2140625924293047,0.7672985979628816);
			\coordinate (3) at (0.9,0.734059609640115);
			\coordinate (4) at (0.6533351621279061,0.9609518342781311);
			\coordinate (5) at (0.26024789140115234,0.277384813553407);
			\coordinate (6) at (0.26024789140115234,0.7759696383949077);
			\coordinate (7) at (0.8,0.35);
			\draw[edge] (4) to (7)  (1) to (3)  (5) to (6)  (1) to (6)  (3) to (7)  (2) to (5)  (3) to (5)   ;
			\draw[edge] (1) to (2)  (4) to (6)  (5) to (7)  (3) to (6)  (1) to (7)  (2) to (3)  (3) to (4)  (2) to (4) ;
			\node[lnode] at (1) {1};
			\node[lnode] at (2) {2};
			\node[lnode] at (3) {3};
			\node[lnode] at (4) {4};
			\node[lnode] at (5) {5};
			\node[lnode] at (6) {6};
			\node[lnode] at (7) {7};
			\end{tikzpicture}& 
			\begin{tikzpicture}[scale=0.9]
			\cubeCoord
			\draw[edge] (1) to (2)  (4) to (7)  (2) to (6)  (4) to (5)  (1) to (4)  (1) to (3)  (2) to (3)  (3) to (7)  ;
			\draw[edge] (2) to (5)  (3) to (5)  (2) to (7)  (6) to (7)  (4) to (6)  (1) to (5)  (3) to (6)   ;
			\node[lnode] at (1) {1};
			\node[lnode] at (2) {2};
			\node[lnode] at (3) {3};
			\node[lnode] at (4) {4};
			\node[lnode] at (5) {5};
			\node[lnode] at (6) {6};
			\node[lnode] at (7) {7};
			\end{tikzpicture}\\ 
		\end{tabular}
	\end{center}
	\caption{All 7-vertex graphs constructed only by an H2 move in the last step.}
	\label{fig:7vertexH2}
\end{figure}

Using Henneberg steps, we tackle the case $n=7$. There are $18$ graphs constructed using a sequence of only H1 steps, 
while two are obtained if we apply H1 to $G_{16}$. Hence, the number of real embeddings is 16, resp.\ 32, by the doubling argument.
Moreover, there are 6 graphs obtained by H2 on a $6$-vertex graph, see Figure~\ref{fig:7vertexH2}.
See~\cite{sourceCode} for the full list.

The results (mixed volume for both systems, numbers of complex and real embeddings) for these 6 graphs are in Table~\ref{tab:sevenVertResults}. 
These results give a full classification of the embeddings of all 7-vertex minimally rigid graphs in $\mathbb{R}^3$.
We present edge lengths for all these graphs proving that all embeddings can be real, i.e., $r_3(G)=c_3(G)$.

\begin{table}[!htb]
	\begin{center}
		\begin{tabular}{ccccccc}
			\textbf{Graph} & $G_{48}$ & $G_{32a}$ & $G_{32b}$ & $G_{24}$ & $G_{16a}$ & $G_{16b}$ \\ \hline
			\rule{0cm}{1.em}MV sphere eq. & 48 & 32 & 32  & 32  & 32 & 32 \\ 
			MV dist. subsyst. & 48 & 32 & 32  & 24  & 24 & 16 \\ 
			$c_3(G)$& 48 & 32 & 32  & 24  & 16 & 16 \\
			$r_3(G)$ & 48 & 32 & 32  & 24  & 16 & 16 \\
		\end{tabular}
	\end{center}
	\caption{Mixed volume (MV) and number of solutions for 7-vertex graphs constructed only by H2 in the last step.}
	\label{tab:sevenVertResults}
\end{table}

There are 20 subgraphs of $G_{48}$ given by vertices $(u,v,w,p,c)$ satisfying the assumption in Lemma~\ref{lem:couplerCurvePreserves},
that is, they are suitable for the sampling procedure.	
Using tree search approach, we obtained $\mathbf{d}$ such that $r_3(G_{48},\mathbf{d})=48$  in only 3 steps
(starting from~$\mathbf{\bar{d}}$ and using subgraphs $(v_5, v_6, v_1, v_7, v_4)$, $(v_4, v_3, v_1, v_7, v_5)$ and $(v_3, v_2, v_1,v_ 7, v_4)$):\par\nobreak
{\parskip0pt \footnotesize \noindent
	\begin{align*}
	d_{12} &= 1.9999 , & d_{16} &= 2.0001 , & d_{45} &= 7.0744 , & d_{47} &= 11.8471 ,  \\
	d_{13} &= 1.9342 , & d_{26} &= 1.0020 , & d_{56} &= 4.4449 , & d_{57} &= 11.2396 ,  \\
	d_{14} &= 5.7963 , & d_{23} &= 0.5500 , & d_{27} &= 10.5361 , & d_{67} &= 10.5365\,. \\
	d_{15} &= 4.4024 , & d_{34} &= 5.4247 , & d_{37} &= 10.5245 , 
	\end{align*}
}
For other graphs constructed only by an H2 move in the last step we used various starting lengths, 
we just list the edge lengths that give the appropriate maximal number of real embeddings:\par\nobreak
{\parskip0pt \footnotesize \noindent
	\begin{align*} 
	\rule{0em}{1.05em} G_{16a}: & &
	d_{13} &= 5.75, & d_{56} &= 7.90, & d_{16} &= 8.48, \\
	d_{37} &= 5.91, & d_{25} &= 7.15, & d_{35} &= 5.09, & d_{12} &= 4.36, \\
	d_{46} &= 8.78, & d_{57} &= 10.22, & d_{36} &= 7.06, & d_{17} &= 3.77, \\
	d_{47} &= 7.19, & d_{23} &= 3.81, & d_{34} &= 3.23, & d_{24} &= 6.05 \,.\\
	\rule{0em}{1.05em} G_{16b}:& &
	d_{47} &= 4.46, & d_{26} &= 7.47, & d_{45} &= 7.72, \\
	d_{14} &= 6.51, & d_{13} &= 3.53, & d_{23} &= 7.69, & d_{37} &= 5.76, \\
	d_{25} &= 9.48, & d_{35} &= 6.10, & d_{12} &= 4.62, & d_{67} &= 3.09, \\
	d_{27} &= 5.90, & d_{46} &= 7.07, & d_{15} &= 5.69, & d_{36} &= 6.43 \,.\\
	\end{align*}
	\begin{align*}
	\rule{0em}{1.05em} G_{24}:& &
	d_{47} &= 5.65, & d_{26} &= 5.70, & d_{56} &= 4.70, \\
	d_{14} &= 8.33, & d_{13} &= 4.77, & d_{23} &= 10.31, & d_{37} &= 7.10, \\
	d_{25} &= 9.32, & d_{12} &= 11.05, & d_{46} &= 6.49, & d_{57} &= 5.77, \\
	d_{27} &= 6.00, & d_{15} &= 9.40, & d_{36} &= 8.57, & d_{34} &= 7.64 \,.\\
	\rule{0em}{1.05em} G_{32a}:& &
	d_{13} &= 6.27, & d_{56} &= 9.23, & d_{14} &= 8.06, \\
	d_{23} &= 8.83, & d_{37} &= 5.62, & d_{25} &= 9.74, & d_{35} &= 5.60, \\
	d_{12} &= 10.95, & d_{67} &= 9.28, & d_{57} &= 7.88, & d_{36} &= 8.26, \\
	d_{47} &= 8.74, & d_{16} &= 11.56, & d_{34} &= 6.11, & d_{24} &= 8.95 \,. \\
	\rule{0em}{1.05em} G_{32b}:& &d_{47} &= 85.49, & d_{26} &= 7.11, & d_{56} &= 22.08, \\
	d_{14} &= 87.33, & d_{13} &= 10.81, & d_{23} &= 4.47, & d_{37} &= 7.10, \\
	d_{25} &= 20.70, & d_{12} &= 11.06, & d_{67} &= 9.29, & d_{15} &= 21.49, \\
	d_{27} &= 7.68, & d_{45} &= 78.53, & d_{36} &= 7.53, & d_{34} &= 84.17 \,.
	\end{align*}
}

\subsection{8-vertex graphs}
We repeated our methods for $n=8$. There are $311$ graphs that can be
constructed by an H1 step (hence, $r_3(G)$ is known by H1 doubling
argument), while $63$ require an H2 step.  So we computed only the
complex bounds of the latter: $58$ of them have $96$ complex
embeddings or less, one has $112$ complex embeddings, $3$ have $128$
complex embeddings and there is a unique graph $G_{160}$ with $160$
complex embeddings.

\begin{figure}[!htb]
	\begin{tabular}{cc}
		$G_{160}$ & $G_{128}$ \vspace{1mm}\\
		\begin{tikzpicture}[yscale=0.65, xscale=0.82]
		\maxEightVert
		\draw[edge] (1) to (2)  (2) to (7)  (4) to (7)  (2) to (6)  (6) to (8)  (4) to (5)  (2) to (8)  (5) to (7)  (3) to (4) ;
		\draw[edge] (1) to (4)  (1) to (5)  (1) to (3)  (1) to (6)  (5) to (6)  (3) to (7)  (7) to (8)  (2) to (3)  (5) to (8) ;
		\node[lnode] at (1) {1};
		\node[lnode] at (2) {2};
		\node[lnode] at (3) {3};
		\node[lnode] at (4) {4};
		\node[lnode] at (5) {5};
		\node[lnode] at (6) {6};
		\node[lnode] at (7) {7};
		\node[lnode] at (8) {8};
		\end{tikzpicture}  &
		\begin{tikzpicture}[yscale=0.75, xscale=0.82]
		\ringEightVert
		\draw[edge] (2)edge(3)  (3)edge(4) (7)edge(2) (5)edge(4) (6)edge(5) (7)edge(6);
		\draw[edge] (2)edge(1) (1)edge(4) (1)edge(3) (1)edge(5) (1)edge(6) (1)edge(7);
		\draw[edge] (2)edge(8) (8)edge(4) (8)edge(3) (8)edge(5) (8)edge(6) (8)edge(7);
		\node[lnode] at (1) {1};
		\node[lnode] at (2) {2};
		\node[lnode] at (3) {3};
		\node[lnode] at (4) {4};
		\node[lnode] at (5) {5};
		\node[lnode] at (6) {6};
		\node[lnode] at (7) {7};
		\node[lnode] at (8) {8};
		\end{tikzpicture}  
	\end{tabular}

	\caption{$G_{160}$ has the maximal number of complex embeddings (160). We proved that $r_3(G_{128})=128$.}
	\label{fig:G160and128}
\end{figure}
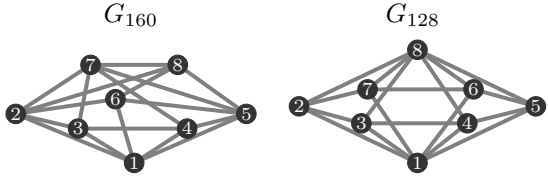

We were interested in improving the lower bound established previously.
No graph with less than $128$ embeddings could improve the bound obtained by $G_{48}$.
Thus, we applied the technique to maximize real embeddings on two different 8-vertex graphs:
$G_{160}$ and $G_{128}$, see Figure~\ref{fig:G160and128}. 
Both can be constructed by H2 step from~$G_{48}$, and the structure of  $G_{128}$ is similar to $G_{48}$.
Therefore, we use some lengths of $G_{48}$ with many embeddings for the common edges for the starting lengths.
The following edge lengths of $G_{128}$ with 128 real embeddings were found by the algorithm:\par\nobreak
{\parskip0pt \footnotesize \noindent
	\begin{align*}
	d_{12} &= 8.7093 , & d_{17} &= 2.1185 , & d_{68} &= 10.5532 , & d_{56} &= 0.7536 , \\
	d_{13} &= 10.3433 , & d_{28} &= 13.5773 , & d_{78} &= 10.5509 , & d_{67} &= 1.5449 ,  \\
	d_{14} &= 1.9373 , & d_{38} &= 14.6173 , & d_{23} &= 13.5267 , &  d_{27} &= 9.2728\,. \\
	d_{15} &= 1.9379 , & d_{48} &= 10.5237 , & d_{34} &= 10.1636 , \\
	d_{16} &= 2.0691 , & d_{58} &= 10.5237 , & d_{45} &= 0.0634 ,
	\end{align*}
}
For  $G_{160}$ we have obtained lengths for 132 real embeddings:\par\nobreak
{\parskip0pt \footnotesize \noindent
	\begin{align*}
	d_{12} &= 1.999 , &d_{23} &= 1.426, &d_{37} &= 10.447, &d_{58} &= 4.279, \\
	d_{13} &= 1.568 , &d_{26} &= 0.879, &d_{45} &=  7.278, &d_{68} &= 0.398, \\
	d_{14} &= 6.611, &d_{27} &= 10.536, &d_{47} &= 11.993, &d_{78} &= 10.474\,. \\
	d_{15} &= 4.402, &d_{28} &= 0.847,, &d_{56} &= 4.321, &\\
	d_{16} &= 1.994, &d_{34} &= 6.494, &d_{57} &= 11.239, &
	\end{align*}
}
More values for edge lengths of the presented graphs with various
number of embeddings are available in \cite{sourceCode}.
We remark that it takes only few seconds to construct all Geiringer graphs up to 8 vertices and compute their mixed volumes.
Complex embeddings computation takes approximately 1 second for one 7-vertex graph and 4 seconds for an 8-vertex graph.
Although we take advantage of tracking solutions in our implementation of the sampling method
(which speeds up the computation significantly), the time of sampling for $G_{160}$ takes about 8 hours using 8 cores (this time strongly depends on the starting lengths).
In order to get the lengths with 132 embeddings, we tested about 200 different starting conformations.

\subsection{Lower bounds}

To compute a lower bound on the maximum number of
embeddings for rigid graphs in the space with $n$ vertices, we use as
a building block a rigid graph $G$ with low $|V_G|$, but high $r_3(G)$.
To do so, we need the following theorem from~\cite{GraKouTsiLower17}:

\begin{thm}
	Let $G$ be a rigid graph,  with a rigid subgraph $H$.
	We construct a rigid graph using $k$ copies of $G$, where all the copies have the subgraph~$H$ in common.
	The new graph is rigid, has $n=|V_H|+k(|V_G|-|V_H|)$ vertices, and
	the number of its real embeddings is at least
	\begin{equation*}
	2^{(n-|V_H|)\!\mod(|V_G|-|V_H|)} \cdot r_3(H) \cdot
	\left(\frac{r_3(G)}{r_3(H)}\right)^{\left\lfloor\frac{n-|V_H|}{|V_G|-|V_H|}\right\rfloor}.
	\end{equation*}
\end{thm}

If we use $G_{160}$ as $G$ and one of its triangle subgraphs as $H$, then we obtain the following lower bound.
\begin{cor}
	The maximum number of real embeddings of rigid graphs in $\RR^3$ with $n$ vertices is bounded from below by
	\begin{equation*}
	2^{(n-3) \mod   5} \, 132^{\lfloor (n-3)/5 \rfloor}  \,.
	\end{equation*}
	The bound asymptotically behaves as $2.6553^{n}$.
\end{cor}

The previous lower bound was $2.51984^n$ \cite{Emiris1}, whereas using
$G_{48}$ as $G$ gives $2.6321^{n}$.  It would be tempting to use as
$H$ a tetrahedron, say $T$, for which it holds $r_3(T) = 2$.  Such a choice of a subgraph would have further
improved the lower bound as the denominator of the exponent would have
been smaller.  Unfortunately, $G_{48}, G_{128}$, and $G_{160}$ do not contain a tetrahedron as a
subgraph.

\section{Conclusion}
By exploiting the (semi-)algebraic modeling of the embeddings of minimally
rigid spatial graphs we present new classification results
and a novel method to maximize
$r_3(G)$. The latter led to improved lower bounds.
Finding better asymptotic bounds is always an open question.  A first
step should be to find out the exact value of $r_3(G_{160})$.  
Furthermore, computations using our method for
$n\geq 9$ may also give better results in this direction.  Another
direction is to find an efficient variant of our sampling method
for other dimensions.  Finally, subsystems of the determinantal
varieties may improve the upper bounds for $c_3(G)$.

\paragraph{Acknowledgments}
This work is part of the project ARCADES that has received funding
from the European Union's Horizon~2020 research and innovation
programme under the Marie Sk\l{}odowska-Curie grant agreement
No~675789.  ET is partially supported by ANR JCJC GALOP
(ANR-17-CE40-0009) and the PGMO grant GAMMA.

\vspace{15pt}
\bibliographystyle{plain}
\bibliography{rigid}

\end{document}